\documentclass[12pt]{amsart}
\usepackage{latexsym,amsmath,amsfonts,amscd,amssymb}
\usepackage{graphics}
\newtheorem{theorem}{Theorem}[section]
\newtheorem{lemma}[theorem]{Lemma}
\newtheorem{proposition}[theorem]{Proposition}
\newtheorem{corollary}[theorem]{Corollary}

\newtheorem{definition}[theorem]{Definition}
.

\theoremstyle{remark}

\newcommand{\Id}{\operatorname{Id}}

\newcommand{\Li}{\operatorname{Li}}
\newcommand{\li}{\operatorname{li}}

\newcommand{\Slim}{\operatorname{S-lim}}
\newcommand{\barstar}{\operatorname{\, \bar \star\, }}
\newcommand{\lcm}{\operatorname{lcm}}

\newcommand{\cF}{{\mathcal F}}

\newcommand{\cO}{{\mathcal O}}

\newcommand{\CC}{{\mathbb C}}

\newcommand{\QQ}{{\mathbb Q}}

\newcommand{\ZZ}{{\mathbb Z}}

\renewcommand{\a}{\alpha}
\renewcommand{\b}{\beta}

\newcommand{\g}{\gamma}

\newcommand{\eps}{\epsilon}

\begin{document}
\title[Local monodromy formula of Hadamard  products]
{Local monodromy formula of Hadamard  products}


\author[R. P\'{e}rez-Marco]{Ricardo P\'{e}rez-Marco}

\address{Ricardo P\'erez-Marco\newline 
\indent  Institut de Math\'ematiques de Jussieu-Paris Rive Gauche, \newline
\indent CNRS, UMR 7586, \newline
\indent Universit\'e de Paris, B\^at. Sophie Germain, \newline 
\indent 75205 Paris, France}

\email{ricardo.perez-marco@imj-prg.fr}

\begin{abstract}
\noindent
We find an explicit general formula for the iterated local monodromy of singularities of the Hadamard product of 
functions with integrable singularities.  The formula implies the 
invariance by Hadamard product of the class of functions with integrable singularities with recurrent monodromies.
In particular, it implies the recurrence of the local 
monodromy of functions with finite Hadamard grade as defined by Allouche and Mend\`es-France. We give other examples 
of natural classes of functions with recurrent monodromies, functions with algebro-logarithmic singularities, 
and more generally with polylogarithm monodromies.
We sketch applications to elliptic integrals, hypergeometric functions, and to  fractional integration.
\end{abstract}

\maketitle

\bigskip

\begin{figure}[ht] \label{fig:mareacion}
\centering
\resizebox{8cm}{!}{\includegraphics{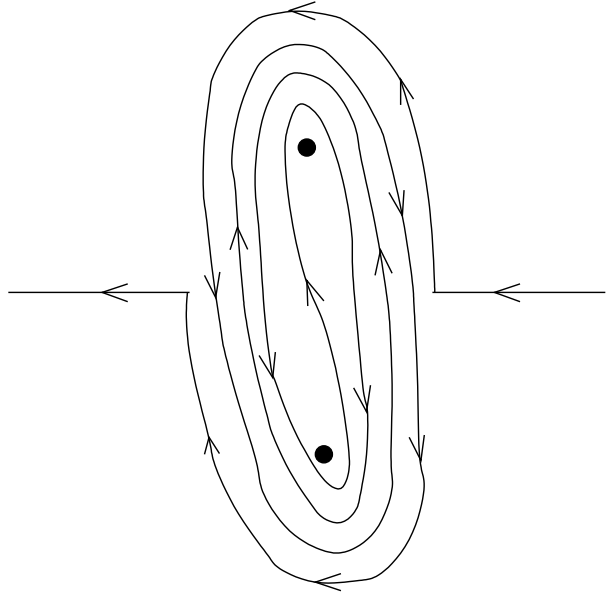}}    
\caption{\footnotesize{``El mundo de la mareaci\'on'' (2003, drawings from my 5 years old son Riki).}} 
\end{figure}

\medskip

\newpage

\section{Introduction.} \label{sec:introduction}

Given  two power series
\begin{align*}
F(z) &=A_0+A_1 z+A_2 z^2+\ldots =\sum_{n\geq 0} A_n \, z^n\\ 
G(z) &=B_0+B_1 z+B_2 z^2+\ldots =\sum_{n\geq 0} B_n \, z^n
\end{align*}
their classical Hadamard product  is defined by 
$$
F\odot G (z) =A_0B_0+A_1 B_1 z+A_2 B_2 z^2+\ldots =\sum_{n\geq 0} A_n B_n \, z^n \ .
$$

In a previous article \cite{PM2020} we gave a formula for the monodromies of the  singularities of 
the Hadamard product $F\odot G$ of two holomorphic 
functions $F$ and $G$ with isolated singularities $(\a)$, resp. $(\b)$, with 
holomorphic monodromies $\Delta_\a F$ and $\Delta_\b G$. More precisely, $\Delta_\a F= F_1-F$ where $F_1$ 
is the branch obtained as the local analytic continuation of $F$ around $\a$ (following a loop of winding number $1$ with respect to $\a$), 
and the monodromy $\Delta_\a F$ is holomorphic if it defines a holomorphic germ at the point $\a$. 
We consider also the operator $\Sigma_\a=\Id+\Delta_\a$, thus $F_1=\Sigma_\a F = F+\Delta_\a F$.

The singularity
at $\a$ is integrable, when we have
$$
\Slim_{z\to \a} (z-\a) \,  F (z) \to 0
$$ 
where the limit is a Stolz limit at $\a$, that is, $z\to \a$ with $\arg (z-\a)$ bounded. In this case 
we have 
$$
\int_{C_\epsilon} F(u) du \to 0
$$
for a loop $C_\epsilon$ around $\a$ with $C_\epsilon \to 0$ when $\eps\to 0$.
We prove the following result:

\begin{theorem}[Hadamard iterated monodromy formula for integrable singularities]
\label{thm:Hadamard_integrable_iterated_monodromy}
We consider $F$ and $G$ holomorphic germs at $0$ with respective set of 
singularities $(\alpha)$ and $(\beta)$ in $\CC$. We 
assume that the singularities are isolated and integrable. For $k\geq 0$, we consider the different branches around the singularities 
$F_k = \Sigma_\a^k F$ and $G_k = \Sigma_\b^k G$.
Then the 
set of singularities of the principal branch of 
$F\odot G$ is contained in the product set $(\gamma)=(\alpha \beta)$ and is 
composed by isolated singularities,
and for $N\geq 1$ we have the formula
\begin{equation}\label{eq:general_monodromy_convolution}
(\Sigma^N_{\g}-\Id) (F\odot G) (z)= - \frac{1}{2\pi i}  \sum_{\substack{\a ,\b \\ \a\b=\g}}  
\sum_{k=0}^{N-1}\int_\a^{z/\b} \Delta_{\a} \big (\Sigma_\a^k F\big ) (u) \, . \, \Delta_{\b} \big ( \Sigma_\b^k G\big ) (z/u) \, \frac{du}{u}
\end{equation}
\end{theorem}

For $N=1$ this is the same formula from \cite{PM2020} except for the absence of the residual part, and generalized to integrable singularities (it was proved there 
for holomorphic singularities, which are not necessarily integrable and can produce a residual part).

We use the following notation for the integral convolution\footnote{We use $\barstar$ instead of $\star$ that we reserve for the e\~ne product 
that we consider in \cite{PM2020}.}
$$
F\barstar  G =- \frac{1}{2\pi i}  \int_\a^{z/\b}  F (u) \, . \, G (z/u) \, \frac{du}{u}
$$
Then, the main formula can be written as
\begin{equation*}
(\Sigma^N_{\g}-\Id) (F\odot G) (z)=  \sum_{k=0}^{N-1}  \sum_{ \a\b=\g}  
\Delta_{\a} \Sigma_\a^k F \barstar \Delta_{\b}  \Sigma_\b^k G
\end{equation*}
We have another form of the formula that follows from formula (\ref{eq:general_monodromy_convolution}) and 
the formal identity
$$
\Delta \Sigma^k= \Sigma^{k+1}-\Sigma^k = (\Sigma^{k+1}-\Id)- (\Sigma^{k}-\Id)
$$

\begin{proposition}
We have
$$
\Delta_\g \Sigma_\g^k (F\odot G ) =  \sum_{\a\b=\g}  
\Delta_{\a} \Sigma_\a^k F \barstar \Delta_{\b}  \Sigma_\b^k G
$$
Therefore, the operators $(\Delta_\g\Sigma_\g^k)_{k\in \ZZ}$ define morphism from the Hadamard algebra to 
the integral  convolution algebra in the associated monodromy spaces.
\end{proposition}

As Corollary, by induction, we get the formula for the iterated monodromy for a Hadamard product with several factors:

\begin{corollary} We have
\begin{equation}\label{eq:multiple_general_monodromy_convolution}
(\Sigma^N_{\g}-\Id) (F_1\odot F_2 \odot\ldots \odot F_n)= \sum_{k=0}^{N-1}\sum_{\a_1\ldots \a_n=\g}  
\Delta_{\a_1} \Sigma_{\a_1}^k F_1 \barstar \ldots \barstar \Delta_{\a_1} \Sigma_{\a_n}^k F_n
\end{equation}
\end{corollary}

\section{Proof of the iterated monodromy formula.}

\subsection{Geometric proof.}
The proof follows the same lines than the proof given in \cite{PM2020} by starting with Plancherel-Hadamard convolution formula
\begin{equation} \label{convolution}
F\odot G (z)=\frac{1}{2\pi i} \int_\eta F(u)G(z/u)\, \frac{du}{u}
\end{equation}
where $\eta$ is a positively oriented circle centered at $0$ of radius $r>0$ with $|z|/R_G <r< R_F$, 
where $R_F$ and $R_G$ are the respective radii of convergence of $F$ and $G$.

Hadamard Multiplication Theorem is derived from the convolution formula. The singularities of $F\odot G$ are located at points $\g=\a\b$ 
where $\a$ and $\b$ are singularities of $F$ and $G$ respectively. Also from the covolution formula we get 
that the singularities $\g$ are integrable.

 We can consider the case with trivial multiplicity $1$, i.e. when there is only one pair of singularities $(\a,\b)$ that 
satisfies $\g=\a\b$. The general case is a superposition of this case. We can also assume $\a=\b=\g=1$ to simplify the notation. 
We follow the analytic continuation of $F\odot G$ by moving the point $z$ and deforming the path of integration homotopically into $\eta_z$ 
so that the point $z$ never crosses it.

\begin{figure}[h] \label{fig:onemonodromy1}
\centering
\resizebox{10cm}{!}{\includegraphics{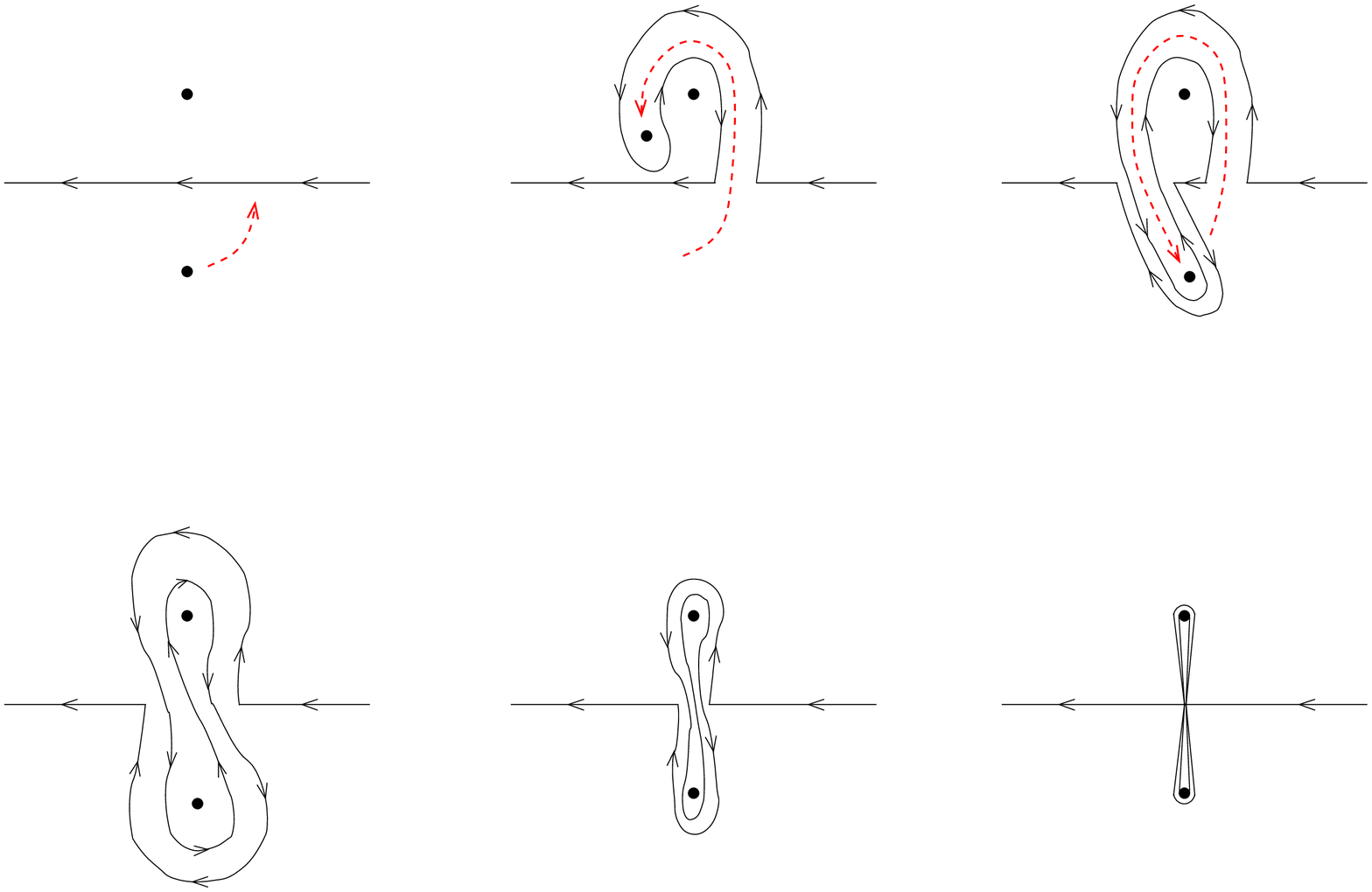}}    
\put (-280,140) {\scriptsize $\eta_{z_0}$}
\put (-248,170) {\scriptsize $1$}
\put (-248,120) {\scriptsize $z_0$}
\put (-158,140) {\scriptsize $\eta_z$}
\put (-142,157) {\scriptsize $1$}
\put (-152,151) {\scriptsize $z$}
\put (-70,138) {\scriptsize $\hat \eta_{z_0}$}
\put (-40,157) {\scriptsize $1$}
\put (-28,124) {\scriptsize $z_0$}
\put (-247,49) {\scriptsize $1$}
\put (-246,11) {\scriptsize $z_0$}
\put (-142,66) {\scriptsize $1$}
\put (-142,6) {\scriptsize $z_0$}
\put (-40,62) {\scriptsize $1$}
\put (-40,10) {\scriptsize $z_0$}
\put (-46,41) {\scriptsize $a$}
\caption{\scriptsize Homotopical deformation of the integration path when $z_0$ turns once around $1$.} 
\end{figure}

We consider the situation where $z$ starts at $z=z_0$ and turns positively $N\geq 1$ times around the point $z=1$. Figure 2 shows the 
local homotopy deformation from the integration path $\eta_{z_0}$ to $\hat \eta_{z_0}$ near $1$  for $N=1$. 
Figure 1 shows the resulting deformed path for $N=2$. Using the convolution formula
we have
$$
(\Sigma_1^N -\Id) (F\odot G) (z_0)=\frac{1}{2\pi i} \int_{\hat \eta_{z_0}-\eta_{z_0}} F(u)G(z/u)\, \frac{du}{u}
$$
Thus we are led to compute the last integral in the difference of the two paths that is the multi-loop path shown in Figure 1. As illustrated in Figure 2, 
we can deform into a train track composed by $4N$ segments copies of $[z_0,a]$ and $[a,1]$ where $a\in \CC$ is an arbitrary point in between $z_0$ and $1$. 

The integral is independent of the collapsing point $a$ and we can make $a\to 1$, leaving a train track composed by $2N$ segments copies of 
$[z_0,1]$. There are no polar contributions at $z=z_0$ or $z=1$ since we assume that the singularities are integrable. 

Considering the $2N$ paths, in the positive ordering given by the orientation of $\hat \eta_{z_0}$, and the corresponding signs given by the orientation, 
we have to integrate on $[1,z_0]$ the maps 
$$
F_1G_0, -F_1 G_1, F_2G_1, -F_2 G_2,\ldots, F_NG_{N-1}, -F_NG_N, F_{N-1}G_N, -F_{N-1}G_{N-1},\ldots ,F_0 G_1, -F_0 G_0 
$$
for a total contribution of
\begin{align*}
&F_1 (G_0-G_1)+F_2 (G_1-G_2)+\ldots +F_N(G_{N-1}-G_N) + F_{N-1}(G_N-G_{N-1})+\ldots  +F_0 (G_1-G_0) \\
&=- F_1 \Delta_1 G_0-F_2 \Delta_1 G_1+\ldots -F_N\Delta_1 G_{N-1} + F_{N-1}\Delta_1 G_{N-1} +\ldots +F_0 \Delta_1 G_1 \\
&=- \Delta_1 F_0 \Delta_1 G_0 - \Delta_1 F_1 \Delta_1 G_1  -\ldots - \Delta_1 F_{N-1} \Delta_1 G_{N-1}
\end{align*}
and finally 
$$
(\Sigma_1^N -\Id) (F\odot G) (z_0)=-\frac{1}{2\pi i} \int_1^{z_0} \sum_{k=0}^{N-1} \Delta_1 F_k(u).\Delta_1 G_k(z/u)\, \frac{du}{u}
$$
which is  the main formula for $\a=\b=1$.$\diamond$

\section{Properties of monodromy operators and another derivation.}

We develop some general properties of the monodromy operators that provide a 
second derivation of the iterated monodromy formula from the case $N=1$. This section is only used 
for this purpose and can be skipped in a first reading. The formulas developed here are useful for 
applications (for example we get a straightforward computation of the monodromy of polylogarithms).

\subsection{Properties of monodromy operators for integrable singularities.} 

\begin{proposition}[Fundamental integro-monodromy formula]\label{prop_fundamental}
Let $f$ be a holomorphic germ with an integrable isolated singularity at $z=\a$. 
The local germ with an isolated singularity at $\a$
$$
F(z)=\int_\a^z f(u) \, du 
$$
is well defined. We have 
$$
\Delta_\a F(z)=\int_\a^z \Delta_\a f(u) \, du 
$$
or
$$
\Delta_\a \left (\int_\a^z f(u) \, du \right )=\int_\a^z \Delta_\a f(u) \, du 
$$
In general, for $z_0\in \CC$ which is path connected to $\a$ in a domain where $f$ is holomorphic, we have
$$
\Delta_{\a} \left (\int_{z_0}^z f(u) \, du \right )=\int_\a^z \Delta_\a f(u) \, du 
$$
\end{proposition}

\begin{proof}
The general case for $z_0\in \CC$ follows from the case $z_0=\a$ since 
$$
\int_{z_0}^z f(u) \, du = \int_{z_0}^\a f(u) \, du +\int_{\a}^z f(u) \, du
$$
where the integral from $z_0$ to $\a$ is taken on a path where $f$ is holomorphic, and 
$$
\int_{z_0}^\a f(u) \, du
$$
is constant on $z$.

Now, we can integrate along the straight segment $[\a, z]$,
$$
F(z)=\int_{[\a, z]} f(u) \, du 
$$
then 
$$
F_1(z) =\int_{[\a, z]\cup \eta_z} f(u) \, du
$$
where $\eta (z)$ is the circle centered at $\a$ passing through $z$ with the positive orientation. Then we have
$$
\Delta_\a F(z)=\int_{\eta(z)} f(u) \, du  \ .
$$
Now, we can collapse homotopically $\eta(z)$ into the segment $[\a, z]$, which gives a double integral along this segment,
$$
\Delta_\a F(z)=\int_\a^z (f_1(u)-f(u)) \, du=\int_\a^z \Delta_\a f(u) \, du   \ .
$$
\end{proof}

\textbf{Example.}
We recall that the polylogarithms can be defined as $\Li_1(z)=-\log(1-z)$ and
$$
\Li_{k+1}(z)=\int_0^z \Li_k(u)\, \frac{du}{u} \ . 
$$
From the properties of the logarithmic function, we have 
$$
\Delta_1 \Li_1(z) = -2\pi i
$$
and by induction using the Proposition \ref{prop_fundamental} we prove
$$
\Delta_1 \Li_{k} (z) = -2\pi i \frac{(\log z)^{k-1}}{(k-1)!} 
$$
The one line proof by induction is the computation
$$
\Delta_1 \Li_{k+1} (z) = \int_1^z \frac{-2\pi i}{(k-1)!} \, \frac{(\log u)^{k-1}}{u} \, du =-\frac{2\pi i}{k!} \left [(\log u)^k\right ]_1^z =-\frac{2\pi i}{k!} (\log z)^k 
$$
which gives the result.$\diamond$

\medskip

We extend Proposition \ref{prop_fundamental}.

\begin{proposition}\label{prop_integro_monodromy_2_var}
Let $u\mapsto f(u,z)$ be a holomorphic germ with an integrable isolated singularity at $u=\a$, depending holomorphically on $z$ in a 
neighborhood $U$ of $\a$ with $(u,z)\mapsto f(u,z)$ holomorphic for $(u,z)$ in a neighborhood of 
$\bigcup_{z\in U} \bar B(\a,z)\times \{z\}-\{(z,z)\}$. We assume that $u\mapsto f(u,z)$ is uniformly integrable at $u=z$ (i.e. along paths landing at $u=z$).
The local germ with an isolated singularity at $\a$
$$
F(z)=\int_\a^z f(u, z) \, du 
$$
is well defined. We have 
$$
\Delta_\a F(z)=\int_\a^z \Delta_{[\eta(z)]} f(u,z) \, du 
$$
where $\eta (z)$ is the loop starting at $z$ whose support is the circle centered at $\a$ passing through $z$ with the positive orientation.
\end{proposition}

Note that  $f(u,z)$ may have an integrable singularity at $u=z$, as it will be the case in our main application of this formula.
The proof follows from the same argument as for Proposition \ref{prop_fundamental} since we can integrate along $\eta(z)$, and 
then we collapse homotopically $\eta(z)$ onto the segment $[\a, z]$.

\medskip

We recall the notation for the integral convolution
$$
f\, \bar \star\,  g =- \frac{1}{2\pi i}  \int_\g^{z}  f (u/\b) \, . \, g (\b z/u) \, \frac{du}{u} \  .
$$

\begin{corollary}
Consider $f$, resp. $g$,  having an isolated singularity at $\a$, resp. $\b$, and $\g =\a\b$. We have
$$
\Delta_\g (f\, \bar \star\,  g) = \Delta_\a f\, \bar \star\,  g + f\, \bar \star\, \Delta_\b  g + \Delta_\a f\, \bar \star\, \Delta_\b  g  
$$
\end{corollary}

\begin{proof}
We only need to compute the monodromy along $\eta(z)$ and apply Proposition  \ref{prop_integro_monodromy_2_var}. Observe that when the variable $u$ goes along 
$\eta(z)$, then $u/\b$ winds once around $\a$, and $\beta z/u$ winds once around $\b$, hence 
$$
\Sigma_{[\eta(z)]} f (u/\b) \, . \, g (\b z/u) = f_1 (u/\b) \, . \, g_1 (\b z/u) 
$$
thus 
$$
\Delta_{[\eta(z)]} f (u/\b) \, . \, g (\b z/u) = \Delta_\a f (u/\b) \, . \, g (\b z/u) +  f (u/\b) \, . \, \Delta_\b g (\b z/u) + \Delta_\a f (u/\b) \, . \, \Delta_\b g (\b z/u)
$$
and the result follows.
\end{proof}

\begin{corollary}\label{cor_computation}
We have
$$
\Delta_\g \big ( \Delta_\a F \, \bar \star \, \Delta_\b G \big ) = \Delta_\a \Sigma_\a F \, \bar \star\,  \Delta_\b \Sigma_\b G  - \Delta_\a F \, \bar \star \, \Delta_\b G 
$$
\end{corollary}

\begin{proof}
We write 
$$
\Delta_\a F \, \bar \star \, \Delta_\b G=- \frac{1}{2\pi i}  \int_\g^{z}  \Delta_\a F (u/\b) \, . \, \Delta_\b G (\b z/u) \, \frac{du}{u}
$$
and we use Proposition \ref{prop_integro_monodromy_2_var} applied to the function $f(u,z)=\Delta_\a F(u) \Delta_\b G(z/u)$, and the singulatity at $\g$,
\begin{align*}
\Delta_\g \big ( \Delta_\a F \, \bar \star \, \Delta_\b G \big ) &= \Delta_\a^2 F \, \bar \star \,  \Delta_\b  G + \Delta_\a F \, \bar \star \,  \Delta_\b^2  G 
+\Delta_\a^2 F \,  \bar \star \, \Delta_\b^2  G\\
&= \Delta_\a \Sigma_\a F \, \bar \star \,  \Delta_\b \Sigma_\b G  - \Delta_\a F \, \bar \star \, \Delta_\b G 
\end{align*}
recalling that $\Sigma = \Id +\Delta$.
\end{proof}

\subsection{Derivation from the case $N=1$.}

We assume Theorem \ref{thm:Hadamard_integrable_iterated_monodromy} for $N=1$ and we argue by induction. 
To simplify we assume that we don't have higher multiplicity
for the singularity $\g=\a\b$, i.e. that there is only one pair $(\a, \b)$ such that $\g=\a\b$. The general case 
is obtained by linear superposition.
We write 
\begin{align*}
\Sigma_\g^{N+1} (F\odot G) &= \Sigma_\g^{N} (F\odot G) + \Delta_\g \Sigma_\g^{N} (F\odot G) \\
&= \Sigma_\g^{N} (F\odot G) + \Delta_\g (F\odot G) + \sum_{k=0}^{N-1} \Delta_\g \big ( \Delta_\a \Sigma^k_\a F \,\bar\star\,  \Delta_\b \Sigma^k_\b G \big ) \\
&= \Sigma_\g^{N} (F\odot G) + \Delta_\g (F\odot G) + \sum_{k=0}^{N-1}  \Big (\Delta_\a \Sigma^{k+1}_\a F \,\bar\star\,  \Delta_\b \Sigma^{k+1}_\b G - 
\Delta_\a \Sigma^{k}_\a F \,\bar\star\,  \Delta_\b \Sigma^{k}_\b G  \Big )\\
&= \Sigma_\g^{N} (F\odot G) + \Delta_\g (F\odot G) + \Big (\Delta_\a \Sigma^{N}_\a F \,\bar\star\,  \Delta_\b \Sigma^{N}_\b G - \Delta_\a F\, \bar \star \,  \Delta_\b G \Big )\\
&= \Sigma_\g^{N} (F\odot G) + \Delta_\a \Sigma^{N}_\a F \,\bar\star\,  \Delta_\b \Sigma^{N}_\b G
\end{align*}
where we use the induction hypothesis on the second line, we use Corollary \ref{cor_computation} in the third line, we telescope the sum on the fourth line, and we use  
$\Delta_\g (F\odot G)= \Delta_\a F\, \bar \star \,  \Delta_\b G$  (which is the result for $N=1$) in the last line.
Now, the result follows by using the induction hypothesis in the last equation.$\diamond$

Obviously, if we knew the morphism property for $\Delta_\g \Sigma_\g^N$ we could shorten the proof using this in the first line. The morphism property is 
equivalent to the iterated monodromy formula.

\section{Applications.}

\subsection{Local monodromies of Hadamard products of algebraic functions.}

We consider holomorphic germs $H$ that are the Hadamard product of a finite number of algebraic functions. 
The Hadamard grade, defined by Allouche and Mend\`es-France in \cite{AMF}, is the minimum integer $n\geq 1$ such that 
$$
H=F_1\odot \ldots \odot F_n
$$
where $F_1, \ldots , F_n$ are algebraic functions. The grade is infinite if there is no such decomposition. The Hadamard grade 
is natural when we study the problem of generating the tower of special functions by Hadamard products \textit{\`a la Liouvile}.
We observe that for an algebraic function $F$, for any ramification singularity $\a$, we have that $\Sigma_\a^k F$, for $k\in \ZZ$, 
are the conjugates of $F$, and $(\Sigma_\a^k F)_k$ is $d$-periodic where $d$ is the degree of $F$ (the local minimal period divides $d$).

We consider in the rest of this section algebraic functions $F_1, \ldots F_n$ which have all the singularities integrable.
The main iterated monodromy formula puts a heavy restriction on the local monodromies of the singularities 
of the function $H=F_1\odot \ldots \odot F_n$.

\begin{theorem}
Let $H= F_1\odot \ldots \odot F_n$ where, for $1\leq j\leq n$, the $F_j$ are 
algebraic functions of degree $d_j\geq 2$ with integrable ramification points, then, if $d=d_1\ldots d_n$, we have
for $N\geq 1$, 
$$
(\Sigma_\g^{N+2d} -2\, \Sigma_\g^{N+d} +\Sigma_\g^{N} \big )H =0
$$
\end{theorem}

\begin{proof}
We start from formula (\ref{eq:multiple_general_monodromy_convolution}) 
$$
(\Sigma^N_{\g}-\Id) (F_1\odot F_2 \odot\ldots \odot F_n)= \sum_{k=0}^{N-1}\sum_{\a_1\ldots \a_n=\g}  
\Delta_{\a_1} \Sigma_{\a_1}^k F_1 \barstar \ldots \barstar \Delta_{\a_1} \Sigma_{\a_n}^k F_n
$$
which shows that if $N=kd+r$ the Euclidean division by $d$, $0\leq r<d$, we have
\begin{equation*}
(\Sigma^N_{\g}-\Id) H= \big (k \left (\Sigma_\g^d-\Id\right ) +\left (\Sigma_\g^r-\Id\right ) \big ) H 
\end{equation*}
Therefore, we get $\big (\Sigma_\g^{N+d} -\Sigma_\g^N \big ) H= \big (\Sigma_\g^d -\Id)  H$ and 
$$
(\Sigma_\g^{N+2d} -2\, \Sigma_\g^{N+d} +\Sigma_\g^{N} \big )H =0 
$$
\end{proof}

As Corollary from the proof we get,

\begin{corollary} We have the existence of the following limit
$$
\lim_{N\to +\infty} \frac{1}{N}(\Sigma^N_{\g}-\Id)  (F_1\odot \ldots \odot F_n) =
\frac{1}{d_1\ldots d_n} \sum_{k=0}^{d-1}\sum_{\a_1\ldots \a_n=\g} \Delta_{\a_1} \Sigma_{\a_1}^k F_1 
\barstar \ldots \barstar \Delta_{\a_n} \Sigma_{\a_n}^k F_n  
$$ 
\end{corollary}

This is interesting once we recognize $\frac{1}{N}(\Sigma^N_{\g}-\Id)$ as a Birkhoff sum for $\Delta_\g$ associated 
to the local monodromy dynamics of $\Sigma_\g$ 
$$
\frac{1}{N}\left ( \Sigma^N_{\g}-\Id \right ) =\frac{1}{N} \sum_{k=0}^{N-1} (\Sigma_\g^{k+1}-\Sigma_\g^{k}) = 
\frac{1}{N} \sum_{k=0}^{N-1} \Delta_\g \circ \Sigma_\g^{k} 
$$
We can be more precise and prove a recurrence result of the local monodromy. The proof is done by induction.

\begin{theorem}[Recurrence of the local monodromy]
For $1\leq j\leq n$, let $F_j$ be algebraic functions of degree $d_j\geq 2$ with integrable ramification points. Let 
$D_j = d_{n-j+1}\ldots d_n$ and consider the successive euclidean divisions starting with $N\geq 1$,
\begin{alignat*}{3}
N & =K_n D_n +R_n  && 0\leq R_n <D_n \\
R_n & =K_{n-1} D_{n-1} +R_{n-1} \ \ \ &&0\leq R_{n-1} <D_{n-2}  \\
& \ \ \ \vdots && \ \ \ \ \vdots \\
R_2 &=K_{1} D_{1} +R_1 &&0\leq R_{1} <D_{1} \\
R_1 &= K_0 
\end{alignat*}
so, the $K_i$ for $i<n$ are uniformly bounded on $N$, $0\leq K_i < D_{i+1}$, 
and we have the $(D_n)$-adic decomposition
$$
N=K_n D_n +K_{n-1} D_{n-1} +\ldots K_1 D_1 +K_0 \ .
$$
Then we have
\begin{align*}
&(\Sigma_\g^N -\Id ) (F_1 \odot \ldots \odot F_n) = K_n \sum_{\a_1\ldots \a_n=\g} \sum_{k=0}^{D_n-1} \Delta_{\a_1} \Sigma_{\a_1}^k F_1 
\barstar \ldots \barstar \Delta_{\a_n} \Sigma_{\a_n}^k F_n \, + \ldots \\
& + K_{n-1} \sum_{\a_1\ldots \a_n=\g} \sum_{l_1=0}^{d_1-1} \Delta_{\a_1} \Sigma_{\a_1}^{l_1} F_1  
\barstar \sum_{k=0}^{D_{n-1}-1} \Delta_{\a_2} \Sigma_{\a_2}^k F_2 \barstar \ldots \barstar \Delta_{\a_n} \Sigma_{\a_n}^k F_n \, + \ldots \\
& + K_{n-2} \sum_{\a_1\ldots \a_n=\g} \sum_{\substack{0\leq l_1 <d_1\\ 0\leq l_2<d_2}} \Delta_{\a_1} \Sigma_{\a_1}^{l_1} F_1  
\barstar \Delta_{\a_2} \Sigma_{\a_2}^{l_2} F_2  \barstar  
\sum_{k=0}^{D_{n-2}-1} \Delta_{\a_3} \Sigma_{\a_3}^k F_3 \barstar \ldots \barstar \Delta_{\a_n} \Sigma_{\a_n}^k F_n \, + \ldots  \\
& \ \ \ \vdots \\
&  + K_{1} \sum_{\a_1\ldots \a_n=\g}\sum_{\substack{0\leq l_1 <d_1\\ \vdots \\ 0\leq l_n<d_n}} \Delta_{\a_1} \Sigma_{\a_1}^{l_1} F_1  
\barstar \ldots \barstar \Delta_{\a_n} \Sigma_{\a_n}^{l_n} F_n  \,+ \ldots  \\
&  + \sum_{\a_1\ldots \a_n=\g}\sum_{k=0}^{K_0} \Delta_{\a_1} \Sigma_{\a_1}^k F_1 \barstar \ldots \barstar \Delta_{\a_n} \Sigma_{\a_n}^k F_n
\end{align*}
\end{theorem}

\subsection{Class of recurrent monodromy.}

We generalize the previous result to functions with recurrent monodromy.

\begin{definition}
An isolated singularity $\a$ of a germ $F$ has a recurrent monodromy of order $d\geq 1$ if 
the iterated monodromy $\Sigma^d_\a$ is a linear combination of monodromies $(\Sigma_\a^k)_{0\leq k\leq d-1}$, that 
is, there are constants $a_k\in \CC$, $0\leq k \leq d-1$, such that, for all $m\geq 0$,
$$
\Sigma^{m+d}_\a F= \sum_{k=0}^{d-1} a_k \, \Sigma^{m+k}_\a F  \ \ .
$$
The \textit{recurrent monodromy class} is the set of functions with only isolated singularities with recurrent monodromies at all singularities.
\end{definition}

A first example are algebraic function which are recurrent of the order equal to their degree and $a_0=1$ and $a_j=0$ for $j=1, \ldots , d$. 
From the result of the previous section, the  Hadamard product of algebraic functions with integrable 
singularities is in the recurrent monodromy class (but it is not necessarily algebraic, see section \ref{sec:elliptic}). 
Observe that in the definition the order of the 
recurrence depends on the singularity.

\begin{theorem}[Invariance of the recurrent monodromy class]\label{th:recurrent_class}
Let $F, G$ be functions in the recurrent monodromy class with integrable singularities. Then their Hadamard product
$F\odot G$
is also in the recurrent monodromy class with integrable singularities. 
\end{theorem}

\begin{proof}
 Let $m(\g)\geq 1$ be the multiplicity of $\g$,
i.e. the number of $(\a, \b)$ such that $\g=\a\b$. Let $d=d_\a d_\b m(\g)$ where $d_\a$ and $d_\b$ are the degree of the 
recurrence for the monodromies of $F$ and $G$ at points $\a$ and $\b$ respectively such that $\g=\a \b$.

From the main formula we have for $k\geq 0$,
$$
\Sigma^{m}_\g F\odot G= F\odot G + \sum_{k=0}^{m-1}\sum_{\a \b=\g} \Delta_{\a} \Sigma_{\a}^k F \barstar \Delta_{\b} \Sigma_{\b}^k G
$$
and plugging the recurrence relation by induction we lower the exponents and end-up with 
$$
\Sigma^{m}_\g F\odot G= F\odot G + \sum_{(k,l)}\sum_{\a\b=\g} c_{m,k,l}(\a,\b) \, \Delta_{\a} \Sigma_{\a}^{k} F 
\barstar \Delta_{\b} \Sigma_{\b}^{l} G
$$
where $0\leq k \leq d_\a$, $0\leq l \leq d_\b$ and constant $c_{m,k,l}(\a,\b)\in \CC$.

If we consider $d+1$ of these relations, for $m=0, \ldots,  d$, then by linear elimination we find a linear combination that kills   
all the $\Delta_{\a} \Sigma_{\a}^{k} F \barstar \Delta_{\b} \Sigma_{\b}^{l} G$, and we can express $\Sigma^{d}_\g H$ linearly in function of 
the lower iterated monodromies  $\Sigma^{m}_\g H$, for 
$m=0, \ldots,  d-1$.
\end{proof}
Observe that if the singularities $\a$ and $\b$ of $F$ and $G$ have 
respective orders $d_\a$ and $d_\b$, then if the singularity $\g =\a \b$ of $F\odot G$ has no multiplicity, then 
the monodromy at $\g$  is $d_\a d_\b$-recurrent.  For $m(\g)\geq 2$ we can improve the above proof observing that parts 
of the sum associated to different pairs $(\a,\b)$ do not interact. Then this gives that the monodromy at $\g$ is $D$-recurrent 
for 
$$
D=\lcm \{d_\a d_b ; (\a,\b), \a\b=\g\}
$$
This shows  a graduated structure of the Hadamard algebra of functions with recurrent monodromies.

An interesting subclass of the recurrent monodromy class are those functions whose 
monodromy is $K$-recurrent for a subfield $K\subset \CC$, for example when $K=\QQ$ 
or $K$ is a number field. 

\begin{definition}[$K$-recurrent monodromy class]
Let  $K\subset \CC$ be a subfield of $\CC$. A recurrent monodromy function has a $K$-recurrent monodromy if we have a recurrence
$$
\Sigma^{m+d}_\a F= \sum_{k=0}^{d-1} a_k \, \Sigma^{m+k}_\a F  \ \ .
$$
with coefficients in $a_k\in K$ for all $0\leq k\leq d-1$. 
\end{definition}

\begin{corollary}
Let  $K\subset \CC$ be a subfield of $\CC$.
The $K$-recurrent monodromy class with integrable singularities is invariant by Hadamard product.
\end{corollary}

\begin{proof}
In the proof of Theorem  \ref{th:recurrent_class} we have all coefficients $c_{m, \boldmath k}(\a_1,\ldots ,\a_n)\in K$ and the linear 
combination can be done with scalars in $K$, hence the result. 
\end{proof}

We can refine the previous results allowing the field $K$ to depend on the singularity, say $K_\a$. 
Then, for the Hadamard product $F\odot G$, the associated field $K_\g$ for the singularity is an extension of the fields $(K_\a)$ and $(K_\b)$
such that $\a\b=\g$.

\textbf{Examples.}

We have some remarkable examples. In all these examples $K$ is a subfield of $\CC$.

\medskip

\textbf{Example 1. Finite dimensional récurrence.} We consider the class of functions $\cF$ having integrable singularities at places $\a\in \CC$, such that for each singularity $\a$ there 
is a finite $K$-dimensional space $V_\a$ of dimension $d_\a \geq 1$, and for $F\in \cF$ and all $k=0,1,\ldots d-1$, $\Sigma_\a^k F \in V_\a$.
Then $\cF$ is a class of functions with $K$-recurrent monodromy so they Hadamard products are in the $K$-recurrent monodromy class.

Given a linear differential equation with rational function coefficients, the set of local solutions at a regular point (out of poles) is a finite dimensional vector space. 
If we have a bases of solutions that are integrable (which is more precise than fuchsian) then all solutions are integrable. Moreover, since the monodromy operator 
$\Sigma_\a$ commutes with differential operators with rational function coefficients (if the coefficients don't have monodromies at the poles or singularities this also works),
then we get that all solutions have recurrent monodromies. Hence their Hadamard products do have recurrent monodromies also.

\medskip

\textbf{Example 2. Algebro-geometric singularities.}

We consider the class of functions $F$ having only algebro-geometric singularities, as those considered in \cite{Ju}. 
This means that $F$ has only isolated singularities and 
locally near a singularity $\a$ we have
$$
F(z)=(z-\a)^{-a_\a} \left (\frac{\log (z-\a)}{2 \pi i} \right )^{n_\a} \varphi_\a(z)
$$
where $a_\a\in \CC$, $n_\a\geq 0$ is a positive integer, and $\varphi \in \cO_\a$ a local 
holomorphic germ. All these parameters depend on the singularity $\a$, but we may drop the sub-index $\a$ to simplify the notation.

These are the type of singularities appearing in solutions of fuchsian equations.
Morevoer, we request that $\Re a_\a> -1$ so that the singularities are integrable.

Observe that we have  for $k\in \ZZ$,
\begin{align*}
\Sigma_\a^k F &=\Sigma_\a^k(z-\a)^{-a} \Sigma_\a^k\left (\frac{\log (z-\a)}{2 \pi i} \right )^{n} \Sigma_\a^k\varphi(z) \\
&=e^{-2\pi i k a} (z-\a)^{-a} \left (\frac{\log (z-\a)}{2 \pi i} +k \right )^{n} \varphi(z)
\end{align*}
So, the monodromies are in the finite dimensional space generated by the germs,
$$
E_l(z)=(z-\a)^{-a} \left (\frac{\log (z-\a)}{2 \pi i} \right )^{l} \varphi (z)
$$
for $0\leq l\leq n$. Using the argument from  Example 1 we conclude that we have a recurrent monodromy. We can be more precise.
For a finite linear combination we have
\begin{align*}
\sum_k c_k \, \Sigma_\a^k F &=(z-\a)^{-a} \varphi (z) \left ( \sum_k c_k \, e^{-2\pi i k a} \left (\frac{\log (z-\a)}{2 \pi i} +k \right )^{n} \right )\\
&= \sum_{l=0}^n \binom{n}{l} \left (  \sum_k c_k \, e^{-2\pi i k a} k^l \right ) \left (\frac{\log (z-\a)}{2 \pi i} \right )^{l}
\end{align*}
Since the $(n+1)\times (n+1)$ Vandermonde determinant (with convention $0^0=1$ here)
$$
\det[k^l] =\prod_{k_1\not=k_2} (k_1-k_2) \not=0
$$
is non-zero, for $N=n$ we can find a non-trivial linear combination ($\delta_{l,n}$ is Kronecker symbol)
$$
\sum_k d_k k^l =\delta_{l,n}
$$
and putting $c_k=d_k e^{2\pi i k a}$ we prove that the monodromy is recurrent of degree $n$. Observe that the monodromy is $K$-recurrent 
if and only if $e^{-2\pi i a} \in K$. Thus, if $K$ is a number field then we must have $\a\in \QQ$.

\newpage

\textbf{Example 3. Polylogarithm singularities.}

This is a generalization of the previous example. It gives examples with recurrent monodromy that are not fuchsian.

We assume that for each singularity $\a$ we have locally near $\a$ that $F$ is a germ of the form 
$$
F(z)= (z-\a)^a L(z) \varphi (z)
$$
where $L(z)\in \CC[\li_1(z-\a), \ldots , \li_n(z-\a)]$ where the $\li_k$ are the normalized poly-logarithms
$$
\li_k(z)=-\frac{\Li_k(z)}{ 2\pi i }
$$
hence
$$
\Delta_1 \li_k (z) = \frac{(\log z)^{k-1}}{(k-1)!}
$$
and 
\begin{equation}\label{eq:polylog_mon}
\Sigma_\a \li_k (z-\a) = \li_k(z-\a)+ \frac{(\log z)^{k-1}}{(k-1)!}
\end{equation}
Since $L$ is a polynomial on polylogarithm functions, say $L=P(\li_1, \ldots , \li_n)$,  we denote by $d$ the 
degree of $P$, and we call $d$ the polylogarithm degree of $L$.
\begin{lemma}
We have
$$
\Sigma_\a^k F =e^{2\pi i ka} (z-\a)^a P\left ( \li_1+k, \ldots, \li_n +k \frac{(\log z)^{n-1}}{(n-1)!} \right )  \varphi (z)
$$
and 
$$
P\left ( \li_1+k, \ldots, \li_n +k \frac{(\log z)^{n-1}}{(n-1)!} \right ) =Q(\log z,\li_1, \ldots , \li_n) 
\in \CC[\log z, \li_1(z-\a), \ldots , \li_n(z-\a)]
$$
and the degree of the polynomial $Q$ is bounded by $(n-1)d$ in the first logarithmic variable, and is the same as the one of $P$ in the other 
variables.
\end{lemma}

\begin{proof}
It is straightforward from the monodromy formula (\ref{eq:polylog_mon}) for polylogarithms. 
\end{proof}

\begin{corollary}
With the previous notations, a function $F$ with polylogarithm singularities is in the recurrent monodromy class. If $P\in K[X_1,\ldots, X_n]$ has 
coefficients in the field $K$, then the monodromy of $F$ at the singularity $\a$ is $K$-recurrent. 
\end{corollary}

\subsection{Elliptic integrals, hypergeometric functions and fractional integration.}\label{sec:elliptic}

In this section we sketch one of the applications of the monodromy formula to elliptic integrals and to hypergeometric functions.
We refer to \cite{Ha} for the classical theory of elliptic integrals, and to \cite{Sla} and \cite{AAR} for the theory of hypergeometric functions.

As is well know, the Hadamard product of algebraic functions is not in general an algebraic function. 
A notable exception happens in finite characteristic
by a Theorem of Deligne \cite{De}. One simple example from \cite{Ju} is 
$$
F(z)=G(z)=(1-z)^{-1/2} = \sum_{n=0}^{+\infty}  \frac{1.3\ldots (2n-1)}{2.4\ldots (2n)}  z^n
$$
which give by Hadamard product the elliptic integral
$$
F\odot G (z) = \sum_{n=0}^{+\infty}  \left (\frac{1.3\ldots (2n-1)}{2.4\ldots (2n)}\right )^2  z^n = \frac{2}{\pi} \int_0^1 \frac{du}{\sqrt{(1-u^2)(1-zu^2)}}
$$
which is proved by developing the right hand side and observing that 
$$
\int_0^1 \frac{u^{2n}}{\sqrt{1-u^2}} \, du = \int_0^{\pi/2} (\cos x)^{2n} \, dx =\frac{1.3\ldots (2n-1)}{2.4\ldots (2n)} \, \frac{\pi}{2}  \ \ .
$$
This Hadamard product is the classical modular function in the variable $k=z^2$ which is not an algebraic function (since it is easy to check that 
the local monodromies are not of finite order),
$$
K(k)=F\odot G (k^2)=\frac{2}{\pi} \int_0^1 \frac{du}{\sqrt{(1-u^2)(1-k^2u^2)}}
$$
We can directly compute the monodromy of the modular function by application of the monodromy formula. Both $F$ and $G$ have an integrable 
singularity at $z=1$, and we get the closed form expression
$$
\Delta_1 (F\odot G) (z) =-\frac{1}{2\pi i} \int_1^z \frac{du}{\sqrt{u(1-u)(u-z)}}
$$
For the modular function, using the change of variable formula (Proposition 3.11 from \cite{PM2020}), we have
$$
\Delta_{\pm 1} K (k)= \Delta_{1}  (F\odot G) (k^2) = -\frac{1}{2\pi i} \int_1^{k^2} \frac{du}{\sqrt{u(1-u)(u-k^2)}}
$$

More generally, we can generate some classical and generalized hypergeometric functions from Hadamard products of simple polar functions, for example,
we have
\begin{equation} \label{eq:Hadamard_polar}
(1-z)^{-a} \odot (1-z)^{-b} = F(a,b,1;z)
\end{equation}
and this generalizes the previous example (where $a=b=1/2$ and $\a=\b=1$).
Then the classical Euler integral formula 
$$
F(a,b,c;z) =\frac{\Gamma(c)}{\Gamma(b) \Gamma(c-b)} \int_0^1 u^{b-1}(1-u)^{c-b-1} (1-zu)^{-a} \, du
$$
can be interpreted as a monodromy integral formula. In fact, for $c=1$ the formula  can be derived 
from equation (\ref{eq:Hadamard_polar}) observing that 
$$
\Delta_1 (1-z)^{-a} =\left (e^{-2\pi i a}-1 \right )(1-z)^{-a}
$$
and from the knowledge of the monodromy of $z\mapsto F(a,b,c;z)$ (see \cite{DLMF} formula 15.2.3 or Theorem 2.3.3 in \cite{AAR})
$$
\Delta_1 F(a,b,c;z)  = \frac{\Gamma(c) \Gamma(a+b-c)}{\Gamma (a)\Gamma (b)} (z-1)^{c-a-b} F(c-a, c-ab, c-a-b+1; 1-z)
$$
that can be obtained directly from the hypergeometric differential equation and the Kummer 
solutions (as in Theorem 2.3.3 in \cite{AAR}). 

The same analysis can be carried out for higher integral formulas of the classical hypergeometric functions as
for example 
$$
F(a,b,c;z) =\frac{\Gamma(c)}{\Gamma(b) \Gamma(c-b)} \int_0^1 u^{b-1}(1-u)^{c-b-1} F(a,b,c; zu) \, du
$$
or for generalized hypergeometric functions
$$
{}_p F_q 
\begin{bmatrix}
a_1,\ldots ,a_p \hfill \\
\ \ \ \ \ \ \ \ \ \ \ \ \ \  ;z \\
b_1,\ldots ,b_q \hfill \\
\end{bmatrix}
=\frac{\Gamma(b_1)}{\Gamma(a_1) \Gamma(b_1-a_1)}  \int_0^1 u^{a_1-1}(1-u)^{b_1-a_1-1} 
{}_{p-1} F_{q-1} 
\begin{bmatrix}
a_2,\ldots ,a_p \hfill   \\
\ \ \ \ \ \ \ \ \ \ \ \ \ \  ;zu \\
b_2,\ldots ,q_p \hfill   \\
\end{bmatrix}
\, du
$$

With this systematic procedure, although computational laborious, we can derive an important number of integral formulas in the 
theory of hypergeometric functions that appear as particular cases of our monodromy formula. Also this shows a link of the convolution operation 
in the monodromy formula with the convolutions appearing in Katz's theory of rigid local systems (see in particular the introduction of 
\cite{Ka}\footnote{I am indebted to J. Fresan for pointing out this reference.}).

More generally, when we consider a finite dimensional vector space of fuchsian functions, as 
those arising as solutions of a fuchsian equation, the local monodromies are linear operators, and the monodromy integral formulas 
provide a multitude of integral relations between these functions. Thus we can see that the integral relations 
in the theory of generalized hypergeometric functions
goes far beyond the hypergeometric class of functions. To develop properly a full theory it is convenient 
to enlarge the Hadamard product to functions with integrable singularities 
at $0$. 

\medskip

One remarkable method of derivation of these integral formulas is trough fractional calculus and fractional integration by parts, as iniciated 
by Erd\'elyi in \cite{Er} (see also \cite{AAR} section 2.9). One explanation for the success of this approach is that iterated 
and fractional integration is another incarnation of the monodromy formula. More precisely, we can compute the iterated integral, 
for an integer $n\geq 1$,
\begin{align*}
I_n(f) (z) &= \int_\a^z \int_\a^{u_1} \ldots \int_\a^{u_{n}} f(u_{n}) \, du_{n} \ldots  du_1\\
&=\frac{1}{(n-1)!}\int_\a^z (z-u)^{n-1} f(u)\, du 
\end{align*}
thus the fractional integral is naturally defined for $\Re \a >0$ as
$$
I_\a (f) (z)= \frac{1}{\Gamma(\a)} \int_\a^z (z-u)^{\a-1} f(u)\, du
$$
and we can write this formula as
$$
I_\a (f) (z)= \frac{-2\pi i } {\Gamma(\a)} \left (-\frac{1}{2\pi i}\int_\a^z  u^{\a}f(u) \left (\frac{z}{u}-1\right )^{\a-1}\, \frac{du}{u} \right ) 
$$
Therefore we have (abusing the theory)
$$
I_\a (f) (z)= -\frac{2\pi i } {\Gamma(\a)} \ \Delta_\a \left (F\odot (z-1)^{\a-1} \right )
$$
where
$$
F(z)= \frac{1}{2\pi i} \log(z-\a) z^\a f(z) 
$$
so that (we are asuming $f$ regular at $\a$)
$$
\Delta_\a F (z) = z^\a f(z)
$$
Again, in order to make sense of this formula for $\Re \a >1$, we need to extend the Hadamard 
product theory to functions with singularities at $0$. Part of such theory can be found in some early articles of S. Mandelbrojt. 
For example, for ramified germs which are series of the form
$$
F(z)=\sum_{\alpha} a_\a z^{\a}
$$
with the proper convergence conditions, 
the Plancherel-Hadamard convolution formula makes sense integrating along an infinite circle centered at $0$ and our derivation 
of the monodromy formula carries out with the same arguments. These developments are left for future work.

In conclusion, we would like to remove the local ``integrability assumption'' in the results. For this we need to find monodromy formulas
for the general class of isolated singularities with monodromy. The nature of the formulas is geometric and this is left for the next article 
on monodromies.

\end{document}